\documentclass[11pt]{article}

\usepackage{amsthm,amsmath, amsfonts, enumerate,pdfsync}

\usepackage[scale=.65]{geometry}
\usepackage{graphicx,subfig}
\usepackage{setspace}
\usepackage{tikz,pgfplots}
\usepackage{thmtools,thm-restate}
\usetikzlibrary{arrows,chains,matrix,positioning,scopes,snakes,pgfplots.groupplots}

\usepackage[scaled=0.92]{helvet}	
\usepackage{mathptmx}

\newtheorem{Proposition}{Proposition}
\newtheorem{Theorem}{Theorem}
\newtheorem{Lemma}{Lemma}
\newtheorem{Corollary}{Corollary}
\newtheorem{Remark}{Remark}
\newcommand{\R}{\mathbb{R}}

\newcommand{\Rp}{\mathbb{R}\mathrm{P}}

\newcommand{\ad}{\operatorname{ad}}
\renewcommand{\ker}{\mathrm{ker\ }}
\newcommand{\E}{\mathbb{E}}

\newcommand{\tr}{\operatorname{tr}}

\newcommand{\Ad}{\operatorname{Ad}}

\newcommand{\Sym}{\mathrm{Sym}}
\renewcommand{\>}{\rangle}
\newcommand{\<}{\langle}
\newcommand{\sL}{\mathcal{L}}
\newcommand{\vspan}{\operatorname{span}}

\newcommand{\ddto}{\left.\operatorname{\frac{d}{dt}}\right|_{t=0}}

\newcommand{\fsun}{\mathfrak{su}(n)}
\newcommand{\fson}{\mathfrak{so}(n)}

\newcommand{\fz}{\mathfrak{Z}}
\newcounter{para}

\usepackage{color}

\title{Geometric methods for optimal sensor design}

\author{M.-A. Belabbas\thanks{University of Illinois, Urbana-Champaign, USA. Email: {\tt belabbas@illinois.edu}}}

\date{}

\begin{document}
\maketitle

\begin{abstract}
{An \emph{observer} is an estimator of the state of a dynamical system from noisy sensor measurements. The need for observers is ubiquitous, with applications in fields ranging from engineering to biology to economics. The most widely used observer is the Kalman filter, which is known to be the optimal estimator of the state when the noise is additive and Gaussian. Because its performance is limited by the sensors to which it is paired, it is natural to seek an optimal sensor for the Kalman filter. The problem is however not convex and, as a consequence, many  ad hoc methods have been used over the years to design sensors. We show in this paper how to characterize and obtain the optimal sensor for the Kalman filter.   Precisely, we exhibit a positive definite operator which optimal sensors have to commute with. We furthermore provide a gradient flow to find optimal sensors, and prove the convergence of this gradient flow to the unique minimum in a broad range of applications. This optimal sensor yields the lowest possible estimation error for  measurements with a fixed  signal to noise ratio. The results presented here also apply to the dual problem of optimal actuator design.}
\end{abstract}

\section{Introduction}

Since the early work of Kalman, Bucy~\cite{kalman1960new, Kalman61newresults} and Stratonovich \cite{stratonovich1960}, the estimation of linear systems has  expanded its range of applications from its engineering roots~\cite{Rao11092001} to fields such as environmental engineering, where for example it is used to estimate sea-level change~\cite{Hay26022013}; financial engineering, where for example it is used to estimate the realized volatility error~\cite{RSSB:RSSB336} or to price energy futures~\cite{manoliu2002energy}; to economics~\cite{athans_kalmanecon1974}, process control~\cite{Musulin2005629} or even biology~\cite{Barnes:2011:10.1098/rsfs.2011.0056}. The common thread to these applications is that one cannot observe exactly all internal variables of a system, but instead needs to estimate them from partial, noisy measurements coming from a set of sensors.

\maketitle
\noindent We address in this paper the optimal design of such sensors.

There are well-developed methods in control theory to design estimators of the state of a dynamical system based on sensor measurements; such  estimators are called \emph{observers} of a system. It stands to reason that as the signal to noise ratio of the measurements increases, the estimation error afforded by an observer decreases. The question of interest is thus to find which measurements with a given signal to noise ratio are optimal  for an observer, optimal in the sense that the estimation error is minimized. Because the Kalman filter is the minimum mean square estimator of the state~\cite{optimalfiltering_anderson79}, optimal measurements for the Kalman filter yield the \emph{lowest estimation} error which one can obtain \emph{for a given signal to noise ratio}. We call the sensor providing such measurements optimal. 

The optimal sensor design problem for Kalman filters is almost as old as the Kalman filter itself, and over the years a variety of methods have been proposed, we refer the reader to the recent thesis~\cite{alan_optimalthesis2011} for a survey. The major obstacle encountered is that the optimization problem, formulated precisely below, defining an optimal sensor is not convex. To sidestep this obstacle, suboptimal  solutions obtained by way of convex relaxations or ad-hoc heuristics for specific application are often used~\cite{optimal_placement_geneticair1991,Musulin2005629}. 
Another approach of choice is to focus on a convex performance measure~\cite{geromel_convex89, darivandi2013algorithm} or optimize  bounds for the estimation error~\cite{li2009designing}. There is also a extensive literature discussing the properties of, and numerical methods for, optimal sensor/actuator placement in  infinite dimensional spaces, see~\cite{fahroo2000optimal, optimal_Actuator_morris2011} and references therein.

In this  paper, we provide an exact characterization of the optimal sensors for Kalman filters by exhibiting a positive definite matrix they have to commute with. We furthermore provide a gradient algorithm---in fact, a Lax equation~\cite{lax1968integrals}--- 
to find such optimal sensors and prove its convergence  to the global optimum in a broad range of situations.  Finally, we demonstrate the efficacy of the methods proposed with simulations and provide a rule of thumb for choosing sensors that work best in low signal to noise ratio settings. We also believe that the  geometric analysis provided here sheds light on the intrinsic difficulty of the problem, difficulty that arises because the constraints on the number of observation signals and their signal to noise ratio are not convex.  The optimal sensor design problem  is equivalent to an optimal actuator placement problem, which we discuss in details below.

Closely related problems which might benefit from the point of view presented here include the optimal scheduling and design of the measurements~\cite{herring1974optimum}, the joint optimal measurement and control design~\cite{bansal1989simultaneous} or the control of complex systems~\cite{li2009designing}.   %

We now describe the problem and our results precisely. We start with a few conventions used throughout the paper. All square matrices are real $n \times n$ matrices unless otherwise specified. We denote by $I_n$ the $n \times n$ identity matrix, by $\Omega_{ij}$ the skew-symmetric matrix with zero entries everywhere except for the $ij$th and $ji$th ones, which are $1$ and $-1$ respectively, and  by $\Sigma_{ij}$ the symmetric matrix with zero entries everywhere except for the $ij$th and $ji$th ones, which are both one. We simply say norm of a vector to refer to its Frobenius norm. For $J$ a differentiable function on a manifold and $X$ a vector field on the same manifold, we let  $X \cdot J = dJ\cdot X$ be the directional derivative of $J$ along $X$. We denote by $\R^+$ the set of strictly positive real numbers. 
\subsection{Optimal sensor design.} Consider the linear stochastic differential equation

\begin{equation}
\label{eq:linsysest}
\left\lbrace \begin{aligned}
 dx(t) &= Ax(t) dt + Gdw(t)\\
 dy(t) &=  cx(t) dt + dv(t)
\end{aligned} \right.
\end{equation} where $w(t)$ and $v(t)$ are independent Wiener processes and  $A \in \R^{n \times n},  G \in \R^{n \times r}, c \in \R^{p \times n}$ and $d \in \R$. The  process $x(t)$ is called the state process and $y(t)$ the observation process. The matrix $c$ is the {\bf sensing} or {\bf observation matrix} of the system.   By an estimator for $x$ is meant a dynamical system with input $y(t)$ and whose state, call it $\hat x$,  is an estimate of $x$. 

The  Kalman filter is the optimal, in the mean-squared sense, estimator of the state $x(t)$ given past observations $y([0,t])$.  Given the matrices $A, G$ and $c$ as above, the Kalman filter in \emph{steady state} is
\begin{equation*}
\label{eq:kalmanfilt}
 d\hat x(t) = A \hat x(t) dt-Kc^\top(dy(t)-c \hat x(t) dt) +bu(t) dt 
\end{equation*}
where the matrix $K$ is the symmetric positive definite solution of the following \emph{Riccati} equation:
\begin{equation}
\label{eq:riccoverror}
 KA^\top  + AK  - K c^\top c K + GG^\top =0.
\end{equation} 
Not all sensing matrices are equal for the purpose of estimation. In fact,  it is not too hard to convince oneself that as the norm of $c$ increases, all other things being equal, the estimation error will decrease~\cite{Wredenhagen1993285}.  Keeping these observations in mind,  it is natural to seek the best sensing matrix  of a given norm. To make the statement more precise, denote by $\mathbb{E}$ the expectation operator. One can show that the gain matrix $K$ is also the steady-state covariance of the estimation error~\cite{optimalfiltering_anderson79} $K = \E\left( (x - \hat x)(x-\hat x )^\top\right)
$ and thus the trace of $K$ is nothing else than the MSE estimation error: $$\tr(K)=\sum_i\E((x_i -\hat x_i)^2).$$ We are thus led to the following optimal sensor design problem: \emph{minimize the trace of $K$, where $K$ obeys Eq.~\eqref{eq:riccoverror} over $c$ of  fixed norm.}

\subsection{Optimal actuator design.} In view of the  well-known duality between observability and controllability, it is not  surprising that the optimal actuator design problem takes a formulation similar to the optimal sensor design's. To wit, consider the linear time-invariant system
\begin{equation}
\dot x = Ax + bu.
\end{equation}
An optimal linear quadratic controller is a controller which minimizes the cost functional 
$$J(x)=\int_0^\infty \left(x(t)^TQx(t) +u^2(t)\right) dt,$$
given $x(0)=x$ and for a user-selected positive definite matrix $Q$. It is known that the optimal controller is a feedback controller of the form $u=-b^\top Kx$ where $K$ obeys the Riccati equation $$A^\top K +KA - Kbb^\top K+Q=0.$$ One can show that starting from an initial condition $x_0$, the ``cost of return to zero''  with the above controller is  $J(x_0)=x_0^\top K x_0$. A simple calculation shows that the expected cost of return to zero for an initial condition  $x$ distributed according to an arbitrary rotationally invariant distribution with density $g(r)dr$, where $r=\|x\|$,  is $$\E J =\tr(K) \int_0^\infty g(r)dr .$$  The question thus arises of finding the actuator $b$ that minimizes the trace of $K$. This actuator is the one returning the system to its desired state with the \emph{least effort on average}. As the norm of $b$ increases, the trace of $K$ decreases and we shall thus fix the norm of $b$. By optimizing  a broader class of functions below, our results also handle non-rotationally invariant distributions on the initial conditions.

\subsection{Main results.}

Having shown that optimal sensor and actuator design can both be cast as minimizing the trace of the positive definite solution of the Riccati equation, we pose the following optimization problem, which slightly generalizes the statement introduced above. We adopt the point of view of sensor design, and thus look for an optimal sensing matrix $c$. Without loss of generality, we will represent a sensing matrix by $\sqrt{\gamma} c$ where $\|c\|=\tr(c c^\top)=p$ and $\gamma >0$.  Throughout this paper, we use the notation $$C := c^\top c.$$ With a slight abuse of language, which will be justified below, \emph{we will refer to both $C$ and $c$ as  observation matrices}.  We call an observation matrix $c$ {\bf orthonormal} if $cc^\top = I_p$, where  $I_p$ is the $p \times p$ identity matrix. We say that $C$ is orthonormal if it can be written as $C=c^\top c$ with $c$ orthonormal. For $p=1$,  every observation vector is orthonormal and  the spectral decomposition of $C$ yields $c$ unambiguously. When $p >1$, $C$  defines an orthonormal $ c \in \R^{p \times n}$ up to a $p$ dimensional rotation, since for any $\Theta \in \R^{p\times p}$ with $\Theta^\top \Theta = I_p$, we have $C=c^\top c = (\Theta c)^\top (\Theta c)$. 

We define the cost function, for $L$ and $Q$ positive definite matrices,
\begin{equation}\label{eq:defJ}
J(\gamma,c) := \tr(LK)
\end{equation}
where   $K$ satisfies the Riccati equation
$$ AK +KA^\top -\gamma Kc^\top c K +Q = 0.$$ Note that $J(\gamma,\Theta c)=J(\gamma,c)$. We call an observation vector $c$ {\bf optimal} if it is a global minimizer of $J(\gamma,c)$ for $\gamma$ fixed, and we call it {\bf extremal} if it is a singular point of $J(\gamma,c)$, but not necessarily a minimum. We let $[A,B]=AB-BA$ be the commutator of two matrices $A$ and $B$.  A square matrix is called \emph{stable} if its eigenvalues have strictly negative real parts. For $c \in \R^{p \times n}$, we denote by $\vspan c$ the subspace of $\R^n$ spanned by the rows of $c$. A subspace $V$ of $\R^n$ is called an invariant subspace of $M \in \R^{n \times n}$ if $MV \subset V$. If $M$ is symmetric positive definite, all its $p$-dimensional invariant subspaces are spanned by $p$ eigenvectors of $M$. We refer to the $p$-dimensional invariant subspace of $M$ spanned by the  eigenvectors corresponding to the $p$ largest  eigenvalues  as  the  {\bf highest} $p$-dimensional invariant subspace of $M$.

 The main results of the paper are summarized  below: 
\begin{enumerate}
\item An observation matrix $c \in \R^{p \times n}$ is extremal if $\vspan c$ is an eigenspace of the positive definite matrix \begin{equation}
\label{eq:defM} M := KRK
\end{equation} 
where  $K$ and $R$ are the positive definite solutions of the equations
$$\begin{aligned}
A^\top K+KA -\gamma KCK +Q &=0\\
(A- \gamma CK) R+R(A- \gamma CK)^\top +L&=0.
\end{aligned}$$ \item\label{it:uniqstat} Generically for $L,Q$ symmetric positive definite, for $p=1$, $\gamma>0$ small and $A$ stable, there is a unique (up to a sign) optimal observation matrix $c$. 
\item\label{it:uniqvect} With the same assumptions as in item~\ref{it:uniqstat}, but for  $p \geq 1$, there is a unique optimal orthonormal observation matrix $C =c^\top c$ of rank $p$; it is such that $\vspan c$ is the highest  $p$-dimensional invariant subspace of $M$. 

\item\label{it:itemdiffeq} With the same assumptions as in item~\ref{it:uniqvect}, the differential equation $$\dot C =[C,[C,M]],$$ with $K,R$ as above converges from a set of measure one of initial conditions to an optimal observation matrix.
\end{enumerate}

\section{Optimal sensor and actuator placement.}\label{sec:gradflow}

\paragraph{On the Riccati equation.} The Riccati equation plays a central in the theory of linear systems, and much has been written about its properties. We only mention here, and  without proof, the facts needed to prove our results. A pair $(A,c)$ is called \emph{detectable} if there exists a matrix $D$ such that $A-c^\top D$ is stable.  If the pair $(A,c)$  is detectable and $Q$ is positive definite, the Riccati equation $A^\top K +KA -KCK +Q =0$  has a unique positive-definite solution. Moreover, this solution is such that $A-CK$ is a stable matrix~\cite{brocket_fdls}. In this paper, we will  restrict our attention to stable matrices $A$, in which case the pair $(A,c)$ is detectable regardless of $c$. We discuss this assumption in the last section. We gather the facts needed in the following result, which is essentially~\cite{delchamps_analytic1984}.
\begin{Lemma}\label{lem:limrK}
Let $A$ be a stable matrix and $Q$ a symmetric positive definite matrix. Let $C \in \R^{n\times n}$ be symmetric positive definite  with $\|C\|=1$ and let $\gamma \geq 0$. Then the positive definite solution $K$ of the Riccati equation
$A^\top K +KA-\gamma KCK +Q = 0$ is analytic with respect to $C$ and $\gamma$.
\end{Lemma}

\paragraph{Real projective space and isospectral matrices.}

Denote by $SO(n)$ the special orthogonal group, that is the set of matrices $\Theta \in \R^{n \times n}$ such that $\Theta^\top\Theta = I_n$ and $\det(\Theta)=1$. We denote by $\mathfrak{so}(n)= \lbrace \Omega \in \R^{n \times n} \mid \Omega^\top = -\Omega \rbrace$  the Lie algebra of $SO(n)$ and we use the notation $$\ad_C A := [C,A]:=CA-AC.$$ Let $\Lambda$ be a diagonal matrix.  We denote by $\Sym(\Lambda)$ the orbit of the special orthogonal group $SO(n)$ acting on $\Lambda$ by conjugation; that is $$\Sym(\Lambda) = \left\lbrace C \in \R^{n \times n} \mid C = \Theta ^\top \Lambda \Theta \mbox{ for } \Theta \in SO(n) \right\rbrace.$$ The set $\Sym(\Lambda)$ is the set of all real symmetric matrices which can be diagonalized to $\Lambda$. We call $\Sym(\Lambda)$ an \textbf{isospectral manifold}. A simple computation shows that its tangent space $T_C\Sym(\Lambda)$ at a point $C$ is the following vector space:
\begin{equation}\label{eq:deftcsym}
T_C\Sym(\Lambda) = \left\lbrace [C,\Omega]=\ad_C \Omega \mid \Omega \in \mathfrak{so}(n) \right\rbrace.
\end{equation}

We will only consider the case here of $\Lambda$ having all entries zero or one.  Since we clearly have that  $\Sym(\Lambda)=\Sym(\Lambda')$ if and only if $\Lambda$ and $\Lambda'$ are conjugate, we can define unambiguously  $\Sym(n,p)$ to be the isospectral manifold with $\Lambda$ having $p$ entries one and $n-p$ entries zero on the diagonal. Note that if $C \in \Sym(n,p)$, then $C^2=C$ and $C$ is of rank $p$. Thus $\Sym(n,p)$ can be thought of as the space of rank $p$ orthogonal projectors in $\R^n$.   
The dimension of $\Sym(n,p)$ is easily seen to be $$\operatorname{dim} \Sym(n,p) = np-p^2.$$
 In particular,  $\Sym(n,1)$ is homeomorphic to the real projective space $\Rp(n-1)$. 

We now define the function (see Eq.~\eqref{eq:defJ}) $$\bar J(\gamma,C): \R^+ \times \Sym(n,p) \longmapsto \R:(\gamma,C) \longmapsto \tr(LK)$$ where $K$ is the positive definite solution to the Riccati equation of Lemma~\ref{lem:limrK}. 
With a slight abuse of notation, we will omit the bar over $J$ and write $J(\gamma,C)$ as well.

\paragraph{The normal metric.} The manifold $\Sym(\Lambda)$ possesses a  natural metric called the normal metric or Einstein metric. The main idea behind the definition of the normal metric, which has already been used  in engineering applications~\cite{Helmke:1994ec,brockett1991dynamical,brockett_diffgeomgradesign93}, is to embed $\Sym(\Lambda)$ in the Lie algebra $\fsun$ and use the so-called Killing form~\cite{Knapp:2002ca} on $\fsun$. Note that because $\Lambda$ is not an element of $\fson$, $\Sym(\Lambda)$ is not an adjoint orbit~\cite{atiyah_convexity82} of $SO(n)$. Furthermore, we wish to include the case of $\Lambda$ having repeated entries, which implies that the operator $[C,\cdot]$ (or $\ad_C$, as defined above) acting on $\fson$ is not invertible. We briefly sketch a construction of the normal metric here that emphasizes the properties we shall need below. We refer the reader to~\cite{atiyah_convexity82, brockett_diffgeomgradesign93} for a more careful construction. 

Denote by $\operatorname{im } \ad_C$ the image of $\ad_C$ and by $\ker \ad_C \subset \fson$ the kernel of $\ad_C$. From the definition of $T_C\Sym(\Lambda)$, we see that $\operatorname{im } \ad_C= T_C\Sym(\Lambda).$ The bilinear operator $$\kappa: \fson \times \fson: (\Omega_1,\Omega_2) \mapsto -\tr(\Omega_1\Omega_2)$$ is symmetric and positive definite. It can thus be used to define  the orthogonal complement $(\ker \ad_C)^\perp$ of $\ker \ad_C$ in $\fson$, which we identify with $\fson/\ker \ad_C$.  Using these facts, we can define the \emph{invertible map} $$\bar \ad_C: \fson/\ker \ad_C \mapsto \operatorname{im} \ad_C.$$ The normal metric $\kappa_n$ is defined, for $X,Y \in T_C\Sym(\Lambda)$, as
 
\begin{equation}
\label{eq:defnormmetric}\kappa_n(X,Y):= -\tr(\bar \ad_C^{-1}X\bar \ad_C^{-1}Y).
\end{equation} One can show that the normal metric is positive definite and non-degenerate. Moreover, we have that 
\begin{equation}
\label{eq:propadnormal}
\ad_C \bar \ad_C^{-1} X = X.
\end{equation}
Another property we shall need below is the $\emph{ad-invariance}$ of the trace, which refers to the following relation:
\begin{equation}\label{eq:defadinv}
\tr((\ad_C\Omega_1) \Omega_2) = - \tr(\Omega_1 \ad_C \Omega_2).
\end{equation}

We conclude this section by  describing an orthonormal basis of $T_C\Sym(n,p)$.
\begin{Lemma}\label{lem:isoorthobasis}
Let $1 \leq p \leq n$ and let $\mathcal{M}=\{m_1,m_3,\ldots,m_p\}$ with  $1 \leq m_1 < m_2 < \cdots <m_p\leq n$, all integers. Denote by $\overline{ \mathcal{M}}$ the complement of $\mathcal{M}$ in $\{1,2,\ldots,n\}$. Let $E\in \Sym(n,p)$ be the matrix with zero entries except for the diagonal entries $(i,i)$, $i \in \mathcal{M}$, which are one, that is $$E=\sum_{i\in {\mathcal{M}}} \Sigma_{ii}.$$ Then the matrices $\frac{1}{\sqrt{2}}\ad_E\Omega_{ij}$ for $i \in \mathcal{M}$ and $j \in\overline{ \mathcal{M}}$ form an orthonormal basis of $T_{E}\Sym(n,p)$.
\end{Lemma}

\begin{proof}
Recall that the tangent space at $E$ is spanned by the matrices $\ad_E\Omega$ for $\Omega \in \fson$. Note that the $\Omega_{ij}$, for $i\neq j$ span $\fson$. Hence, to show that the $\Omega_{ij}$ with $i\in \mathcal{M}$ and $j \in\overline{ \mathcal{M}}$ span the tangent space, it is sufficient to show that
$\Omega_{ij} \in \ker \ad_E$ if and only the conditions $i \in \mathcal{M}$ and $j \in\overline{ \mathcal{M}}$ are not satisfied.  But a short calculation shows that \begin{equation}
\label{eq:liebSO}
[\Sigma_{kk},\Omega_{ij}] = \left\lbrace \begin{aligned} \Sigma_{ij} &\mbox{ if } i=k \\
-\Sigma_{ij} &\mbox{ if } j=k\\
0 &\mbox{ otherwise}.
\end{aligned}\right.\end{equation} We conclude from~\eqref{eq:liebSO} that $\ad_E \Omega_{ij} \neq 0$ only if either $i \in \mathcal{M}, j \in \overline{ \mathcal{M}}$ or $i \in \overline{ \mathcal{M}}, j \in \mathcal{M}$.  Since $\Omega_{ij}=-\Omega_{ji}$, the vectors $\ad_E\Omega_{ij}$ with $i \in \mathcal{M}, j \in \overline{ \mathcal{M}}$ span $T_{E_m} \Sym(n,p)$. We now show that these vectors are orthonormal for the normal metric. Again, a straightforward calculation shows that  $$\kappa_n(\ad_E\Omega_{ij},\ad_E\Omega_{kl})= -\tr(\Omega_{ij}\Omega_{kl}) = 2 \delta_{ik}\delta_{jl}$$ where $\delta_{ik} =1$ if $i=k$ and $0$ otherwise. This proves the claim.
\end{proof}

\subsection{Gradient flow for optimal sensor design}\label{ssec:gradient} We now evaluate the gradient flow of $J=\tr(LK)$ with respect to the normal metric.  Fix $C \in \Sym(n,p)$ such that $(A,C)$ is detectable and let $K$ be the corresponding positive definite solution of the Riccati equation. Recall that the gradient of $J$ evaluated at $C$, denoted by $\nabla J(C)$ obeys the relation~\cite{jost_geometricanalysis}
\begin{equation}
\label{eq:graddef}\kappa_n(\nabla J(C), X) = dJ \cdot X,\mbox{ for all } X \in T_{C}\Sym(\Lambda).
\end{equation} Let $C(t)$, be a differentiable curve in $\Sym(\Lambda)$  defined for $|t|<\varepsilon$ small and such that $C(0)=C$ and $\ddto C(t)=X$.   We can choose $\varepsilon$ small enough so that $(A,C(t))$ is detectable for all $|t|<\varepsilon$. From Lemma~\ref{lem:limrK}, we conclude that  for all such $t$,  there exists a unique positive definite solution $K(t)$ to the algebraic Riccati equation $A^\top K(t)+K(t) A -\gamma K(t)C(t) K(t) +Q =0$ and that the curve $K(t)$ is differentiable in $t$. Then \begin{equation}\label{eq:derivgrad1}
dJ \cdot X = \ddto J(C(t)) = \tr(L\ddto K(t))
\end{equation} Differentiating the Riccati equation, and writing $\dot K$  for $\ddto K$, we obtain

$$ A^\top \dot K+\dot K A -\gamma \dot K C K -\gamma K XK- \gamma K C \dot K  =0.$$ The above equation is a Lyapunov equation~\cite{brocket_fdls}, which we can write as
\begin{equation}
\label{eq:dKd} (A-\gamma CK)^\top \dot K + \dot K (A- \gamma CK) - \gamma K X K = 0
\end{equation} 
and whose solution is  \begin{equation}
\label{eq:defdkint}\dot K= -\gamma\int_0^\infty e^{(A-\gamma CK)^\top t}   K X K e^{(A-\gamma CK)t} dt. \end{equation}

Using the definition of $\kappa_n$ from Eq.~\eqref{eq:defnormmetric}, we obtain by plugging~\eqref{eq:defdkint} into~\eqref{eq:graddef}
\begin{equation*}\tr(\bar \ad_{C}^{-1} \nabla J \bar \ad_{C}^{-1} X ) =\gamma \tr( L\int_0^\infty e^{(A-\gamma CK)^\top t}  K X K e^{(A-\gamma CK)t} dt)\end{equation*}

From~\eqref{eq:deftcsym}, we can write $X=\ad_{C}\Omega$ for some $\Omega \in  (\ker \ad_C)^\perp$. Using the cyclic and  ad-invariance properties of the trace, the last equation can be rewritten as
\begin{equation*}\tr(\bar \ad_{C}^{-1}( \nabla J) \Omega ) = \gamma\tr( \ad_{C}\left[  K \int_0^\infty e^{(A-\gamma CK)t}L\right. \\ \left. e^{(A-\gamma CK)^\top t}dt  K\right] \Omega). \end{equation*} The above equation holds for all $\Omega \in  (\ker \ad_C)^\perp$ and thus
$$\bar \ad_{C}^{-1} \nabla J = \gamma \ad_{C}\left[ K  \int_0^\infty e^{(A-\gamma CK)t}Le^{(A-\gamma CK)^\top t}dt  K\right] + \Delta$$ for some $\Delta \in ((\ker \ad_C)^\perp)^\perp = \ker \ad_C$. Taking $ \ad_{C}$ on both sides of the last relation, we obtain $\nabla J =\gamma \ad_C \ad_C KRK
$where $R$ is the solution of the Lyapunov equation
$(A-\gamma CK)R + R(A-\gamma CK)^\top +L =0$. We summarize these calculations in the following Theorem:

\begin{Theorem}\label{th:equilgradientflow}The gradient flow of the function $J(\gamma,C)=\tr(LK)$ with respect to the normal metric and for $\gamma>0$ fixed is $$\dot C = \gamma [C,[C,M]]$$ where $M=KRK$ and $K$, $R$ obey the equations \begin{equation}\label{eq:eqKR}\begin{aligned}
A^\top K+KA +Q -\gamma KCK&=0\\
(A-\gamma CK)R+R(A-\gamma CK)^\top +L &= 0
\end{aligned}.\end{equation} Moreover, an observation matrix $C \in \Sym(n,p)$ is extremal if   it is an orthogonal projection onto a $p$-dimensional invariant subspace of $M$. Equivalently, an orthonormal observation matrix $c \in \R^{p \times n}$  is extremal  if $\vspan c$ is an eigenspace of $M$.
\end{Theorem}
\begin{proof}
The first part of the statement was  proven above. We thus focus on the second part. Recall that extremal points of $J$ are zeros of its gradient, and thus $C$ is extremal if and only if $[C,M]=0$. Because the positive definite solution of the Riccati equation is such that the matrix is $(A-\gamma CK)$ stable and because  $L$ is symmetric positive definite, we have that $R$ is positive definite and thus so is the product $KRK=:M$.  The result is now a consequence of the fact that  symmetric matrices commute if and only if they have the same invariant subspaces.
\end{proof}
\begin{Remark}
It is tempting to conjecture  that if $c_1$ is an extremal observation vector, and $K_1$ and $R_1$ are the corresponding solutions of Eq.~\eqref{eq:eqKR} above, then any eigenvector of $K_1R_1K_1$ is also extremal. This however is not the case.
\end{Remark}

\subsection{The extremal points of J}\label{ssec:hessian}We have derived in the previous section the gradient of $J$. Because $J$ is a lower-bounded function defined on a compact domain, it is clear that the gradient flow will converge to the set of extremal points of $J$.  However, $J$ is not convex and thus we do not have, a priori, convergence to the global minimum of $J$. We show that for $\gamma$ small $J$ has  a unique minimum, a unique maximum and that the other extremal points are finite in number and saddle points. This shows that, in that regime, \emph{the gradient flow will essentially   converge to the global minimum.} We will discuss in the last section how small $\gamma$ needs to be in practice.

We prove the result in two steps: first, we show that there is a finite number of extremal points for $\gamma$ small and then we evaluate their signatures.  Recall that the signature of an extremal point $C$ of $J$ is a triplet of integers $(n_+,n_-,n_0)$, where $n_+$ (resp. $n_-$ and $n_0$) denotes the number of positive (resp. negative, zero) eigenvalues  of the Hessian of $J$ at $C$. The proof of the first item goes by studying the  parametrized  family of vector fields $$F(\gamma,c) = [C,M].$$ When $\gamma>0$, $F$ and $\nabla J$ clearly have the same zeros.  We then show that $F(0,c)$ has  exactly ${n \choose p}$ zeros and that these zeros persist for $\gamma>0$ small. We denote by $\{C_i(\gamma)\}$, $i \in \mathcal{I}(\gamma)$, the set of zeros of $F(\gamma,C)$, where the index set $\mathcal{I}(\gamma)$ is possibly infinite. 

\paragraph{$J$ has a finite number of extremal points.}Let  $\<\cdot,\cdot\>$ be the normal metric on $\Sym(n,p)$. Recall that the Levi-Civita connection is the unique connection that is \emph{compatible with the metric} and \emph{torsion free}~\cite{jost_geometricanalysis}; we denote it by $\nabla$.

\begin{Proposition}\label{prop:finitenumequ}
Let $A$ be a stable matrix. For $\gamma >0$ small and generically for $Q, L$ positive definite matrices, the function $$J(\gamma,C): \R^+ \times \Sym(n,p) \longmapsto \R: J(\gamma,C)=\tr(LK)$$ where $K$ satisfies the Riccati equation~\eqref{eq:eqKR} has exactly ${n \choose p}$ extremal points. 

\end{Proposition}

\begin{proof}

We introduce the following vector field: $$ F: [0,\infty) \times \Sym(n,p) \longmapsto T\Sym(n,p) : (\gamma,C) \longmapsto [C,M]$$ where $M=KRK$ with $R$ and $K$ obeying Eq.~\eqref{eq:eqKR}. One should think of $F(\gamma,C)$ as a parametrized family of vector fields on $\Sym(n,p)$.    We denote by $K_0$ and $R_0$ the solutions of $$A^\top K+KA+Q=0$$ and $$AR+RA^\top+L=0$$ respectively. Set $M_0:=K_0R_0K_0$. For $\gamma=0$ and generically for $Q$ and $L$ positive definite, we conclude from Lemma~\ref{lem:genQL}  (see Appendix) that $M_0$ has $n$ distinct eigenvalues.  Because symmetric matrices commute if and only if they have the same eigenvectors,  there are exactly ${n \choose p}$  matrices $C \in\Sym(n,p)$  which commute with $M_0$. Thus $\mathcal{I}(0)$ contains ${n \choose p}$ elements, say $\mathcal{I}(0)=\{1,2,\ldots, {n \choose p}\}$. Denote by $C_i(0)$ the corresponding zeros of $F$. 

Recall that  $\nabla F$ is the covariant derivative of $F$ where $\nabla$ is the Levi-Civita connection associated to the normal metric. In order to show that for $\gamma>0$ small enough, $\mathcal{I}(\gamma)=\mathcal{I}(0)$, it is sufficient to show that the   linear map $\nabla F: T_C\Sym(n,p) \longmapsto T_C\Sym(n,p): X \longmapsto \nabla_X F(0,C_i) $ is non-degenerate at the ${n \choose p}$ points $(0,C_i(0))$. 
From Lemma~\ref{lem:covderivF}, and for $\Omega_x$ such that $X = [C,\Omega_x]$ we have
$$\nabla_X F=-\frac{1}{2}\left([M_0,[C,\Omega_x]]+[\Omega_x,[C,M_0]]\right).$$
When $C=C_i(0)$ for some $i$,  the second term vanishes. We are left with 
$$ \nabla_X F(0,C_i) =-\frac{1}{2} [M_0,[C,\Omega_x]].$$
We now show that the  covariant derivative is non-degenerate. For this, we need the two following facts:  first,  for any orthogonal matrix $\Theta \in SO(n)$, the conjugation map $$\Ad_\Theta:\fson \longmapsto \fson: \Omega \longmapsto \Theta^{-1} \Omega\Theta $$ has $\Ad_{\Theta^{-1}}$ for inverse and is consequently surjective onto $\fson$. Second, for arbitrary matrices $\Theta \in SO(n)$ and $A,B \in \R^{n \times n}$, \begin{equation}
\label{eq:thetabracket}
\Ad_\Theta [A,B] = [\Ad_\Theta A , \Ad_\Theta B].\end{equation}   Using these facts, we conclude that $\nabla_XF(0,C_i)$ is non-degenerate if and only if the linear map $$ X \longmapsto \Ad_\Theta \nabla_XF(0,C_i) =[\Theta M_0 \Theta^\top,[\Theta C \Theta^\top,\Omega_x]]$$ is non-degenerate.

Because $M_0$ has  exactly $n$ orthonormal eigenvectors, we can let $\Theta$ be the orthogonal matrix whose columns are  eigenvectors of $M_0$. With this choice of $\Theta$, the previous equation reduces to $$\Ad_\Theta \nabla_XF(0,C_i) = [D,[E,\Omega_x]]=\ad_D\ad_E\Omega_x,$$
where $E$ is a matrix with zero entries except for $p$ diagonal entries which are equal to $1$ and $D$ is a diagonal matrix with the eigenvalues of $M_0$ on its diagonal. A short calculation shows that the commutator $\ad_D A$ of a diagonal matrix with diagonal entries $d_i$ and a matrix $A=(a_{ij})$ has entry $ij$ equal to $a_{ij}(d_i-d_j)$. Since the $d_i$ are distinct, we deduce  that $\ad_D$ is full rank. Thus  $\Ad_\Theta \nabla_XF(0,C_i)$ is of full rank.
\end{proof}

\paragraph{The signature of the extremal points of $J$.} We now evaluate the signature of the extremal points of $J$. We denote by $d^2J$ the \emph{Hessian} of the function $J$ with respect to the normal metric;  it is a symmetric, bilinear form on $T\Sym(n)$  and for the vector fields $X$ and $Y$, it  is given by~\cite{jost_geometricanalysis}
\begin{equation}
\label{eq:defHess}d^2J(X,Y) = X \cdot Y \cdot J - \nabla_X Y \cdot J.
\end{equation}

The choice of connection does not affect the type of extremal points of $J$ of course, but it is convenient to fix a connection for the perturbation argument that will be used below. Also note that the   Hessian can be used to accelerate gradient flows or algebraic equation solvers~\cite{yuan2008step}.

\begin{restatable}{Proposition}{propcomphess}
\label{prop:compHessian} Let $J=\tr(LK)$ be defined as in Theorem~\ref{th:equilgradientflow}. Let $X = \ad_C \Omega_x$ and $Y=\ad_C\Omega_y$ for $\Omega_x,\Omega_y \in \fson$. The Hessian of $J$ with respect to the normal metric is 
\begin{equation}
\label{eq:HessFcond}\begin{aligned}
d^2J(X,Y) = \gamma\tr \big\{[ C,\Omega_x][M,\Omega_y] + [ C,VRK+KWK+KRV]\Omega_y\\
 -\frac{1}{2} [ C,M][\Omega_x,\Omega_y]\big\}
\end{aligned}
\end{equation} where $K$ and $R$ are as in the statement of Theorem~\ref{th:equilgradientflow} and $V$ and $W$ are $$V =\gamma \int_0^\infty e^{(A- CK)t} K[ C,\Omega_x] K e^{(A-\gamma CK)^\top t}dt$$ and 

$$ W=\gamma \int_0^\infty e^{(A-\gamma CK)t} \left( [ C,\Omega_x]K- CV- VC-K[ C,\Omega_x]\right)e^{(A-\gamma CK)^\top t}dt.$$ 
\end{restatable}

We prove Proposition~\ref{prop:compHessian} in the Appendix. The following Corollary makes the analysis of $d^2J$ tractable for $\gamma$ small.

\begin{Corollary}\label{cor:expHess}
Let $K_0$ and $R_0$ be the solutions of \begin{equation}
\label{eq:defKo}
A^\top K + KA +Q = 0
\end{equation} and \begin{equation}
\label{eq:defRo}
A^\top R +RA +L = 0\end{equation} respectively. Let \begin{equation}
\label{eq:defMo}M_0:=K_0R_0K_0.
\end{equation}  For $X=\ad_C\Omega_x$, $Y=\ad_C\Omega_y$, the Hessian of $J$ with respect to the normal metric has the following expansion around $\gamma=0$:

\begin{equation}
\label{eq:HessFexp}\begin{aligned}
d^2J(X,Y) \simeq \gamma\tr\left\{[C,\Omega_x][M_0,\Omega_y]-\frac{1}{2} [C,M_0][\Omega_x,\Omega_y]\right\}+ \gamma^2 T(\Omega_x,\Omega_y)
\end{aligned}
\end{equation}
where  the bilinear form $T$ contains terms of zeroth and higher orders in $\gamma$.
\end{Corollary}

\begin{proof}

From Lemma~\ref{lem:limrK}, we know that for $\gamma$ small, the stabilizing solution $K$ of the Riccati equation can be expressed as $$K=K_0+\mbox{h.o.t. in }\gamma.$$ where h.o.t. are higher order terms in $\gamma$. Recall that $R$ obeys the equation $(A-\gamma CK)R-R(A-\gamma CK)^\top +L=0$. This is a linear equation and thus its solution, when it exists, depends analytically on $\gamma$ and $C$. Hence, similarly as for $K$, we can write for $\gamma$ small $$R=R_0+\mbox{h.o.t. in }\gamma.$$ We conclude from the above two expansions that we have \begin{equation}\label{eq:krkext}M \simeq M_0 + \mbox{h.o.t.}\mbox{ in }\gamma.\end{equation}
Now recall the explicit expression of $d^2J$ at $\gamma C$ derived in Proposition~\ref{prop:compHessian}. A simple calculation shows that the first and last terms of the right hand side of~\eqref{eq:HessFcond} admit the expansions $$\tr\left\{[\gamma C,\Omega_x][M,\Omega_y]\right\} = \gamma\tr\left\{[ C,\Omega_x][M_0,\Omega_y]\right\} +\mbox{h.o.t. in } \gamma$$ and $$-\tr\left\{\frac{1}{2} [\gamma C,M][\Omega_x,\Omega_y]\right\} =-\gamma \frac{1}{2}\tr \left\{ [C,M_0][\Omega_x,\Omega_y]\right\}+\mbox{h.o.t. in } \gamma $$ respectively. The second term however, since  both $V$ and $W$ have order one in $\gamma$, contributes terms of order at least two in $\gamma$. We thus have the expansion announced.
\end{proof}

We proved in Prop.~\ref{prop:finitenumequ} that $J$ had a finite number of extremal points $C_i(\gamma)$ for $\gamma$ small. The proof went by showing that the extremal points of $J(\gamma,C)$ were the same as the zeros  of the vector field $F(\gamma,C)=[C,M]$. The latter could however be easily be obtained at $\gamma=0$. We saw that they were of the form $$C_i(0) = \Theta^\top E_i \Theta$$ where $E_i$ is a diagonal matrix with $p$ entries equal to $1$, and the other entries zero, and $\Theta$ is the orthogonal matrix diagonalizing $M_0$~\eqref{eq:defMo}. The following Corollary allows us to evaluate the signatures of the extremal points $C_i(\gamma)$:

\begin{Corollary}\label{cor:signequivbilinear}
Let $M_0$ be as in Eq.~\eqref{eq:defMo} and $\Theta^\top D \Theta$ be its spectral decomposition.  The signature of $d^2J$ at the extremal point $C_i(\gamma)$, for $\gamma$ small, and $C_i(0) = \Theta^\top E \Theta$ where $E$ is a diagonal matrix with $p$ entries equal to $1$ and $n-p$ zero is the same as the signature of the bilinear form 
\begin{equation}\label{eq:HessDiagform}H: T_E \Sym(n,p)\times T_E \Sym(n,p) \longmapsto \R: (X,Y)\longmapsto \tr\left\{[E,\Omega_x][D,\Omega_y]\right\},\end{equation} where $X = \ad_E\Omega_X$, $Y=\ad_E\Omega_y$ and  provided that $H$ is non-degenerate.
\end{Corollary}

\begin{proof}
Let $C= C_i(\gamma)$ and   $X_1, X_2 \in T_C\Sym(n,p)$ be such that $\ad_C \Omega_i =X_i$ for $i=1,2$ for $\Omega_1,\Omega_2 \in (\ker \ad_C)^\top$. Note that the second term in Eq.~\eqref{eq:HessFexp} came from the expansion of the last term in the Hessian of $J$~\eqref{eq:HessFcond}. For $C_i(\gamma)$  an extremal point, this latter term is zero and thus does not contribute to the approximation given Eq.~\eqref{eq:HessFexp}. Hence the dominating term in the Hessian of $J$ at $C_i(\gamma)$ for $\gamma>0$ small is the following bilinear form on $T_{C_i(\gamma)}\Sym(n,p)$: 
$$\gamma\tr\left\{[C_i(\gamma),\Omega_1][M_0,\Omega_2]\right\}.$$
Because $C_i(\gamma)$ depends continuously on $\gamma$, and because no eigenvalues of $H$ are zero by assumption, a standard argument using continuity of the eigenvalues with respect to $\gamma$ shows that the signatures of the extremal points $C_i(\gamma)$ for $\gamma$ small are the same as the one of $C_i(0)$. Hence  the signature of the above bilinear form is the same as the signature of $$\tr\left\{[C_i(0),\Omega_1][M_0,\Omega_2]\right\}.$$
We can simplify the problem further as follows: as was done in the first part of the proof , let $\Theta$ be the orthogonal matrix whose columns contains the eigenvectors of $M_0$.  The cyclic invariance of the trace,  Eq.~\eqref{eq:thetabracket} and the fact that $\Ad_\Theta$ is an isomorphism on $\fson$  together imply that the signature of $d^2J$ at extremal points is the same as the signature of the bilinear form 
\begin{equation}H: T_E \Sym(n,p)\times T_E \Sym(n,p) \longmapsto \R: (\Omega_1,\Omega_2)\longmapsto \tr\left\{[E,\Omega_1][D,\Omega_2]\right\}\end{equation} as was claimed. 
\end{proof}

It thus remains to evaluate the signature of the bilinear form of Eq.~\eqref{eq:HessDiagform}. We do this first for the case $p=1$.

\paragraph{The case of scalar observations.} We start with the case $p=1$, which corresponds to having a scalar observation signal. 
We prove the following Theorem, which covers item~\ref{it:uniqstat} of the main result.

\begin{Theorem}\label{th:equilanalysis}
Let $A$ be a stable matrix. For $\gamma >0$ small and generically for $L, Q$ positive definite matrices, the function $$J(\gamma,C): \R^+ \times \Sym(n,1) \longmapsto \R: J(\gamma,C)=\tr(LK)$$ where $K$ satisfies the Riccati equation~\eqref{eq:eqKR} has exactly $n$ extremal points with  signatures $(n-1,0,0)$, $(n-2,1,0)$,\ldots, $(0,n-1,0)$ respectively. Moreover, the extremal point of signature $(n-p,p-1,0)$ is the orthogonal projection matrix onto the $p$-th highest invariant subspace of $M$.
\end{Theorem}

\begin{proof}From Prop.~\ref{prop:finitenumequ}, we know that for $\gamma$ small, the function $J$ has exactly $n$ extremal points. We now evaluate the signature of the Hessian at these points.
Let $E_j$ be the matrix with zero entries except for the $jj$th entry, which is one. From Corollary~\ref{cor:signequivbilinear}, it suffices to  to evaluate the signatures of the $n$ bilinear forms obtained by letting  $E=E_j$, for $j=1,\ldots,n$, in $H$ in Eq.~\eqref{eq:HessDiagform} and if these are non-degenerate, they will give us the signatures sought. Assume that the diagonal entries of $D$ are sorted in decreasing order. With this ordering, the signature of $H_{E_j}$ is the same as the signature of $d^2J$ at the extremal point $C_j(\gamma)$, where $C_j(0)=c_jc_j^\top$ and $c_j$ is the eigenvector associated to the $j$th largest eigenvalue of  $M_0$. Recall that from Lemma~\ref{lem:isoorthobasis}, an orthonormal basis of the tangent space of $\Sym(n,1)$ at $E_1$ is given by the commutators or $E_1$ and the $n-1$ matrices $\frac{1}{\sqrt{2}}\Omega_{12}, \frac{1}{\sqrt{2}}\Omega_{13}, \ldots, \frac{1}{\sqrt{2}}\Omega_{1n}$. Note that since $D$ is diagonal, $[D,\Omega_{1j}]=\Omega_{1j}(d_1-d_j)$ and thus

$$H_1(\Omega_{1j},\Omega_{1l}) = (d_1-d_j) \delta_{jl}$$ where $\delta_{jl}=1$ if $j=l$ and zero otherwise. This basis hence diagonalizes $H_1$ and shows that its eigenvalues are $(d_1-d_j)$ and are all positive. Thus the signature at $C_1$ is $(n,0,0)$. Now for the general case of $H_j = \tr\{[E_j,\Omega_1][D,\Omega_2]\}$. An orthonormal basis of the tangent space at $E_j$ is given by the commutators of $E_j$ and $\frac{1}{\sqrt{2}}\Omega_{jl}$ for $j \in \{1,2,\hat j,\ldots, n\}$ where  $\hat j$ indicates that $j$ is ommited from the set. Applying the same approach, we find that the eigenvalues of $H_j$ are $(d_j-d_l)$, for $l \in   \{1,2,\hat j,\ldots, n\}$.  Hence $n-j$ eigenvalues are positive and $j-1$ are negative. Thus the signature of $E_j$ is $(n-j,j-1,0)$. This concludes the proof.
\end{proof}

\paragraph{The case of vector-valued observations.}
We now address  the case $p >1$.  Recall that if $c_1$ and $c_2$ are   $p \times n$ matrices of orthonormal rows  that span the same $p$-dimensional subspace of $\R^n$ then, all other things equal, the estimation error of the corresponding Kalman filters have the same statistical properties.
Consequently, optimization problems involving the statistical properties of the estimation error will, when restricted to orthonormal observation vectors,  have loci of extremal values  and all extremal values in the same locus will yield the same estimation performance. 

We now present some combinatorial facts needed to state the main result of this section. Let $m>0$ be an integer. Recall that a {\bf partition} of $m$ with $p$ parts is given by $p$ positive integers $m_1,\ldots,m_p$ whose sum is $m$. The partition is said to have \emph{distinct parts} or to be a {\bf distinct partition} if the integers $m_i$ are pairwise distinct. We denote by $P(p,m)$ the number of partitions of $m$ into $p$ parts and  by $Q(p,m)$ the number of distinct partitions of $m$ into $p$ parts.  One can show that
$$Q(p,m)= P(m-{p \choose 2},p).$$ See~\cite{wilf_generating89}  for more properties of $P$ and methods to compute it.

Let $d=np-p^2$ denote the dimension of $\Sym(n,p)$. We have the following result:
\begin{Theorem}\label{th:equilpart2}
With the same assumptions as in  Theorem~\ref{th:equilanalysis}, the function $$J: \R^+ \times \Sym(n,p) \longmapsto \R: (\gamma,C) \longmapsto \tr(LK)$$ has ${n \choose p}$ equilibria. For any pair $(n_+,n_-)$ of positive integers such that $n_++n_-=d$, there are $Q(p, n_++\frac{p(p+1)}{2})$ extremal points with index $(n_+,n_-,0)$. In particular,  there are unique extremal points with signatures $(d,0,0)$, $(d-1,1,0)$, $(1,d-1,0)$ and $(0,d,0)$ respectively and no degenerate extremal points.
\end{Theorem}

The proof of Theorem~\ref{th:equilpart2} relies on the following Lemma.
\begin{restatable}{Lemma}{lemmasiggeb}\label{lem:calcsignHessgen}
 Let $E$ be a diagonal matrix with diagonal entries 1 and zero and  let $X = \ad_E\Omega_X$, $Y=\ad_E\Omega_y$ be in $T_E\Sym(n,p)$. Define the bilinear form $$H_E: T_E \Sym(n,p) \times T_E \Sym(n,p) \longmapsto \R: (X,Y) \longmapsto  \tr\left\{[E,\Omega_x][D,\Omega_y]\right\}$$ where $D$ is a diagonal matrix with pairwise distinct entries in decreasing order along the diagonal. Let $m_i, i=1,\ldots,p$ denote the positions of the ones on the diagonal of $E$, i.e. $E=\sum_{i=1}^p \Sigma_{m_i,m_i}$ and $m=\sum_{i=1}^p m_i$.  Then the signature of $H_E$ is $(m-\frac{p(p+1)}{2},np-p(p-1)/2-m,0)$. 
\end{restatable}

The proof of Lemma~\ref{lem:calcsignHessgen} is in the appendix.
Note that the signature of the Hessian is independent of the exact values of the entries of $D$, provided they are pairwise distinct and sorted in decreasing order. We first illustrate  Lemma~\ref{lem:calcsignHessgen} on an example. Set $p=4$ and $n=7$ and take $E$ to be the diagonal matrix $$E=\begin{pmatrix}
1 & 0 & 0 & 0 & 0 & 0 & 0\\
0 & 0 & 0 & 0 & 0 & 0 & 0\\
0 & 0 & 1 & 0 & 0 & 0 & 0\\
0 & 0 & 0 & 1 & 0 & 0 & 0\\
0 & 0 & 0 & 0 & 0 & 0 & 0\\
0 & 0 & 0 & 0 & 0 & 1 & 0\\
0 & 0 & 0 & 0 & 0 & 0 & 0
\end{pmatrix}.
$$ For this particular $E$, $m_1=1,m_2=3,m_3=4, m_4=6$ and thus $m=14$. Hence the Lemma says that the bilinear form $H_E$ has a mixed signature $(4,8,0)$. 

The proof of this Theorem is the same as the proof of Theorem~\ref{th:equilanalysis}, save for the evaluation of the signature of the Hessian. We start by summarizing the major steps  leading to where the two proofs differ. For $\gamma>0$ small, there is a one-to-one correspondence between extremal points of $J(\gamma,\cdot)$ and zeros  $[C,M_0]$ where $C=c^\top c, M_0=K_0R_0K_0$ and $K_0, R_0$ are defined in Eqs.~\eqref{eq:defKo} and~\eqref{eq:defRo}.  Because both $C$ and $M_0$ are symmetric, there are ${n \choose p}$ such zeros, corresponding to choices of $p$ eigenvectors of $M_0$. Finally, we have shown that we can assume, without loss of generality, that $M_0$ is a diagonal matrix and that the Hessian at the extremal points has the same signature as the bilinear form $H_E$ of Eq.~\eqref{eq:HessDiagform}.
\begin{proof}[Proof of Theorem~\ref{th:equilpart2}]

An extremal point of $J$ can thus be characterized by $p$ distinct, positive integers $m_1,\ldots,m_p$, indicating the position of the $p$ entries on the diagonal of $E$ that are equal to $1$, the other entries being equal to 0. From Lemma~\ref{lem:calcsignHessgen}, we know that the signature of $H_E$ is $(m-\frac{p(p+1)}{2},np -m-\frac{p(p-1)}{2},0)$ where $m = m_1+m_2+\ldots+m_p$. From the definition of $Q(p,m)$, we see that the number of extremal points with $n_+$ positive eigenvalues is $Q(p,n_++\frac{p(p+1)}{2})$ as announced.

In particular, the number of extremal points whose Hessian is negative definite (i.e. $n_+=0$) is equal to the number of partitions of $\frac{p(p-1)}{2}$ by $p$ distinct positive integers. There is clearly only one such partition, given by $m_1=1,\ldots, m_p=p$. Similarly, $m_1=1,\ldots m_{p-1}=p-1, m_p = p+1$ is the only partition of $\frac{p(p-1)}{2}+1$ with distinct positive integers. Hence there is also a unique extremal point with $n_+=1$. One can show in the same fashion that there are unique extremal points with $n_-=0$ and $n_-=1$.
\end{proof}

As a corollary of Theorem~\ref{th:equilanalysis} and~\ref{th:equilpart2}, we can show that except for a set of measure zero of initial conditions, the differential equation described in item~\ref{it:itemdiffeq} converges to an optimal observation matrix. Precisely, we have the following result:

\begin{Corollary}\label{cor:globconv}
Let $A$ be a stable matrix and $1 \leq p \leq n$. Let $J:\R^+\times \Sym(n,p) \longmapsto \R: J(\gamma,C)=\tr(LK)$ where $L$ is positive definite and $K$ is the positive definite solution of the Riccati equation
$$A^\top K+KA -KCK+Q=0.$$ Let $R$ be the solution of $$(A-CK)R+R(A-CK)^\top +L=0.$$ For $\gamma>0$ small and generically for $Q,L$ positive definite  the differential equation $$\dot C = [C,[C,M]]$$ with $M=KRK$ converges to a global minimum of $J(\gamma,c)$ from a set of measure one of initial conditions.
\end{Corollary}

\section{Discussion}\label{sec:disc}
We posed and solved the problem of finding the sensor minimizing the estimation error afforded by the Kalman filter.  The methodology proposed is applicable to actuator design as well.
The optimal sensor design problem is a difficult problem in the sense that it is not convex.  We  cast the problem as an optimization problem on an isospectral manifold  and  equipped  this space with a Riemannian metric, called the normal metric. We then  evaluated the gradient and Hessian of the cost function $J$ to be optimized. We have shown that for $\gamma$ small, where $\gamma$ is the norm of the observation vector or sensor, and a stable infinitesimal generator $A$ of the dynamics, the gradient flow  converges with probability one to the global minimum. We have restricted the analysis in this paper to the case of orthonormal sensing matrices. Similar results hold for the general case. They are  technically more involved and we do not elaborate on these here due to space constraints and the fact that  most of the main ideas already appear in the present treatment of the orthonormal case.

We now discuss the role of the assumptions made. The first statement of the main result, which characterizes optimal observation matrices, holds free of the assumptions that $\gamma$ be small and the infinitesimal generator $A$ be stable. The second, third and fourth statements, however, relied on these assumptions. From a practitioner's point of view, how small does $\gamma$ need to be? We can answer this question using Eq.~\eqref{eq:HessFcond} and the proof of Theorem~\ref{th:equilanalysis}: the assumption of $\gamma$ small holds for $\gamma< \gamma^*$ where $\gamma^*$ is the smallest $\gamma$ such that the bilinear form
$$\tr\left\{[ C,\Omega_1][M,\Omega_2]+[ C,VRK+KWK+KRV]\Omega_2\right\}$$ with  $C$ extremal has a zero eigenvalue. Indeed, for $0 \leq \gamma <\gamma^*$, we then know that the above bilinear form has no zero eigenvalues and its signature is the one of the lowest order term. Note that $\gamma^*$ depends on $A$ and $Q$. We show in Fig. 1 simulation results, which show that this assumptions holds for rather large $\gamma$ in general. The curves are obtained as follows. We first set $Q = \frac{1}{2} I_4$. We then sampled four batches of $10^4$  real $4 \times 4$ matrices which are stable and whose eigenvalues with largest real parts were, depending on the batch, $-0.1$, $-0.5$, $-1$, or $-3$ (denoted by $\operatorname{Re} \lambda_{m}$ in the legend.). We obtained the samples by drawing matrices from a Gaussian ensemble and then translated their eigenvalues by adding a multiple of the identity matrix. For each sample matrix, and for $\gamma$ ranging from $10^{-3}$ to $10$ we searched for the zeros of the gradient of $J$ numerically and then checked  whether the Hessian at that zero had a signature given by the dominating term. The curves represent the proportion  of matrices, out of the $10^4$ samples, for which $\gamma< \gamma^*$. For example, about $80\%$ of the matrices with $\operatorname{Re}\lambda_m = -\frac{1}{2}$ were such that $\gamma=4$ qualifies as small. Unsurprisingly, as the eigenvalues of $A$ are further away from the imaginary axis, $\gamma^*$ increases and the proportion of matrices for which $\gamma<\gamma^*$, for $\gamma$ fixed increases as well. Indeed, for $\operatorname{Re}\lambda_m=-3$, close to $100\%$ of matrices are such that $\gamma=4$ qualifies as small.
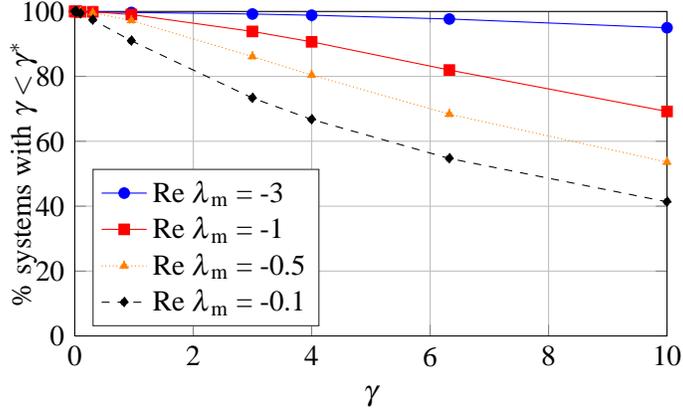
\begin{figure}
\begin{center}
\begin{tikzpicture}
\begin{axis}[%
width=3.1in,
height=1.7in,
scale only axis,
xmin=0,
xmax=10,
xlabel={$\gamma$},
xmajorgrids,
ymin=0,
ymax=100,
ylabel={\% systems with $\gamma < \gamma^*$},
y label style={at={(0.07,0.5)}},
ymajorgrids, 
yminorgrids,
legend style={at={(0.03,0.03)},anchor=south west,legend cell align=left,align=left,draw=white!15!black}
]
\addplot [color=blue,mark=*]
  table[row sep=crcr]{%
0.001	100\\
0.0031	100\\
0.0099	100\\
0.0309	99.87\\
0.097	99.91\\
0.3045	99.85\\
0.9558	99.69\\
3	99.21\\
4	98.8375\\
6.3246	97.675\\
10	94.975\\
};
\addlegendentry{$\text{Re }\lambda{}_\text{m}\text{ = -3}$};

\addplot [color=red,mark=square*]
  table[row sep=crcr]{%
0.001	99.99\\
0.0031	100\\
0.0099	99.99\\
0.0309	99.94\\
0.097	99.9\\
0.3045	99.89\\
0.9558	99.04\\
3	93.84\\
4	90.6375\\
6.3246	81.95\\
10	69.1875\\
};
\addlegendentry{$\text{Re }\lambda{}_\text{m}\text{ = -1}$};

\addplot [color=orange,densely dotted,mark=triangle*]
  table[row sep=crcr]{%
0.001	99.99\\
0.0031	99.98\\
0.0099	99.97\\
0.0309	99.97\\
0.097	99.88\\
0.3045	99.6\\
0.9558	97.2\\
3	86.11\\
4	80.4125\\
6.3246	68.35\\
10	53.5875\\
};
\addlegendentry{$\text{Re }\lambda{}_\text{m}\text{ = -0.5}$};

\addplot [color=black,dashed,mark=diamond*]
  table[row sep=crcr]{%
0.001	99.99\\
0.0031	99.99\\
0.0099	99.95\\
0.0309	99.87\\
0.097	99.38\\
0.3045	97.42\\
0.9558	90.96\\
3	73.33\\
4	66.725\\
6.3246	54.7\\
10	41.35\\
};
\addlegendentry{$\text{Re }\lambda{}_\text{m}\text{ = -0.1}$};

\end{axis}
\end{tikzpicture}%
\end{center}

\caption{The assumption $\gamma$ small holds with high probability for  large value of $\gamma$.}\label{fig:gamma}
\end{figure}

We also assumed that $A$ was stable to reach our conclusions. Note first that Proposition~\ref{prop:compHessian}, which  provide the Hessian of $J$, holds whether $A$ is stable or not.  The assumption  was needed for Lemma~\ref{lem:limrK} to hold when $\gamma = 0$, which in turn allowed us to analyze the Hessian of $J$ via an expansion of the product $M=KRK$ around $\gamma = 0$.  When $A$ is not stable, this expansion does not hold. Furthermore,  it is easy to see that there exist loci of codimension one or two of observation vectors for which $J(\gamma,c)$ is unbounded. Loci of unbounded values can evidently not be crossed by a gradient flow. If the loci are all of co-dimension two, then one might nevertheless have almost global convergence. Even more,  since the domain $\Rp(n-1)$ is not orientable when $n$ is odd,   a locus  of codimension one does not necessarily split the domain  in two disconnected parts. Hence, the analysis of the unstable $A$ case requires a careful analysis of the undetectable modes and the homology class of their eigenspaces. 
{A rule of thumb for sensor choice.}From the proof of Theorem~\ref{th:equilanalysis}, we conclude  that a good observation vector to use is the largest eigenvector of $M_0$ (this matrix is defined in~\eqref{eq:defM}), which we denote by $\gamma c_0$, with $\|c_0\|=1$. Indeed, this vector is optimal for $\gamma=0$ and one can hope that it remains close to optimal as $\gamma$ increases. Note that it is also a good starting point of the gradient flow. In Fig. 2 we present simulation results that show that this is indeed a sensible choice when $\gamma$ is small. The curves in Fig. 2 were obtained as follows: for each curve, we sampled $10^4$ $6 \times 6$  matrices with $\operatorname{Re}\lambda_m$ as indicated on the legend. We let $Q= I_6/\sqrt{6}$.  Denote by $c^*$ be the optimal observer obtained for each sample.  Each curve represents  the average of $J(\gamma,c_0)/J(\gamma,c^*)$ as a function of $\gamma$ for different values of $\operatorname{Re}\lambda_m$. We see that for $\gamma$ very close to $0$, the $c_0$ and $c^*$'s performances are nearly indistinguishable. As $\gamma$ increases, the difference becomes more marked as expected.  We also plotted the performance of a random observer, denoted by $c_r$, which we observe performs predictably worse than both $c^*$ and $c_0$. 

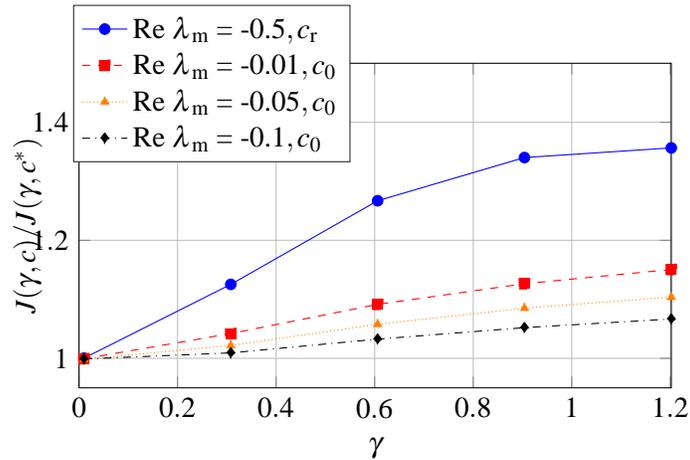
\begin{figure}
\begin{center}
\begin{tikzpicture}
\begin{axis}[%
width=3.1in,
height=1.7in,
scale only axis,
xmin=0,
xmax=1.202,
xlabel={$\gamma$},
xmajorgrids,
ymin=.95,
ymax=1.5,
ylabel={$J(\gamma, c)/J(\gamma, c^*)$},
y label style={at={(0.07,0.5)}},
ymajorgrids, 
yminorgrids,
legend style={at={(-0.01,1.18)},anchor=north west,legend cell align=left,align=left,draw=white!15!black}
]
\addplot [color=blue,mark=*,solid]
  table[row sep=crcr]{%
0.01	1.00024115952667\\
0.308	1.12529466127019\\
0.606	1.26702239268007\\
0.904	1.3402239542061\\
1.202	1.35661875170004\\
};
\addlegendentry{$\text{Re }\lambda{}_\text{m}\text{ = -0.5}, c_{\text{r}} $};
\addplot [color=red,mark=square*,dashed]
  table[row sep=crcr]{%
0.01	0.9993334073270495\\
0.308	1.04194714578447\\
0.606	1.09144448886634\\
0.904	1.12646551471824\\
1.202	1.15055011734086\\
};
\addlegendentry{$\text{Re }\lambda{}_\text{m}\text{ = -0.01}, c_0$};

\addplot [color=orange,mark=triangle*,densely dotted]
  table[row sep=crcr]{%
0.01	0.998309336201366\\
0.308	1.02185772013238\\
0.606	1.05765170179003\\
0.904	1.08501480849734\\
1.202	1.10373344791734\\
};
\addlegendentry{$\text{Re }\lambda{}_\text{m}\text{ = -0.05}, c_0$};

\addplot [color=black,mark=diamond*,dashdotted]
  table[row sep=crcr]{%
0.01	0.998770064538594\\
0.308	1.0092120966289\\
0.606	1.03248764894461\\
0.904	1.05199097532211\\
1.202	1.06660997104321\\
};
\addlegendentry{$\text{Re }\lambda{}_\text{m}\text{ = -0.1}, c_0$};

\end{axis}
\end{tikzpicture}%
\end{center}
\caption{Using $\gamma c_0$ as sensor often  yields a close-to-optimal performance. A random choice of sensor (top curve, $c_r$) performs noticeably worse.}\label{fig:thumb}  
\end{figure}

\bibliographystyle{siam}
\bibliography{sensor_bib}

\appendix
\section{Appendix}

\subsection{The Hessian of $J$ for the normal metric}

We first need an explicit expression for the Levi-Civita connection associated to the normal metric. We will derive such an expression for the case of constant vector fields. We recall that a vector field $X$ in $T\Sym(\Lambda)$ is called a {\bf constant vector field} if it is of the form \begin{equation}
\label{eq:defconstanvf}
X=[C,\Omega_x]\end{equation} for a constant $\Omega_x \in \fson$. 
\begin{Lemma}
\label{lem:levicivitanorma} Let $\<\cdot,\cdot\>$ be the normal metric on $\Sym(\Lambda)$ and let $X=[C,\Omega_x], Y=[C,\Omega_y]$ be constant vector fields in $T\Sym(\Lambda)$. Then the Levi-Civita covariant derivative of $Y$ along $X$  is
$$\nabla_XY = \frac{1}{2} [C,[\Omega_x,\Omega_y]]=\ad_C[\Omega_x,\Omega_y].$$
\end{Lemma}
\begin{proof}Denote by $\mathcal{L}_XY$ the {\bf Lie derivative} of $Y$ in the direction $X$.
Recall that the covariant derivative $\nabla$ obeys the following relation~\cite{jost_geometricanalysis}
\begin{equation}\label{eq:Relconn}
\begin{aligned}\<\nabla_X Y,Z\> = \frac{1}{2} \left[ X\cdot \<Y,Z\>+ Y\cdot \<Z,X\>- Z\cdot \<X,Y\>+\<\mathcal{L}_XY,Z\>\right.\\
\left.-\<\mathcal{L}_YZ,X\>+\<\mathcal{L}_ZX,Y\>\right]
\end{aligned}
\end{equation}
Let  $X=[C,\Omega_x],Y=[C,\Omega_y]$ and $Z=[C,\Omega_z]$ be constant vector fields. A standard calculation shows that $$\mathcal{L}_XY = [C,[\Omega_x,\Omega_y]].$$ Note that $$X \cdot\<Y,Z\> = X\cdot \tr(\Omega_y\Omega_z)=0.$$ We thus have

\begin{align*}
\< \nabla_X Y,Z \> &=  \frac{1}{2} \left[ \<\mathcal{L}_XY,Z\>-\<\mathcal{L}_YZ,X\>+\<\mathcal{L}_ZX,Y\>\right]\\
&= \frac{1}{2} \left[ \<[C,[\Omega_x,\Omega_y]],Z\>-\<[C,[\Omega_y,\Omega_z]],X\>+\<[C,[\Omega_z,\Omega_x]],Y\>\right]\\
&=\frac{1}{2} \left[ \tr([\Omega_x,\Omega_y]\Omega_z)-\tr([\Omega_y,\Omega_z]\Omega_x)+\tr([\Omega_z,\Omega_x]\Omega_y)\right]
\end{align*}
The first two terms cancel each other and we obtain
$$\< \nabla_X Y,Z \> = \frac{1}{2}\tr([\Omega_x,\Omega_y]\Omega_z).$$ Since the previous equation holds for all $\Omega_z \in \mathfrak{so}(n)$ we obtain $$\nabla_X Y = \ad_{C}[\Omega_x,\Omega_y]$$ as announced.
\end{proof}

We recall the statement of Proposition~\ref{prop:compHessian}.
\propcomphess*
\begin{proof}
Let $C \in \Sym(n,p)$ and $X=\ad_C \Omega_x,Y=\ad_C\Omega_y$ be constant vector fields. We start by evaluating the first term in the definition~\eqref{eq:defHess} of the Hessian. From the definition of the gradient and the normal metric, we have that \begin{equation}\label{eq:prlemhess1}
Y\cdot F = \gamma \<[C,M],\Omega_y\>.
\end{equation}
In order to evaluate the differential of the above function along the vector field $X$, we introduce the curve $$C(t) = e^{t\Omega_x}C e^{-t\Omega_x}.$$ We have

$$X \cdot Y \cdot F = \gamma \ddto \<[C(t),K(t)R(t)K(t)],\Omega_y\>.$$

We have already given an explicit expression for $\ddto K(t)$ in Eqn.~\eqref{eq:defdkint}. We now derive an expression for $\ddto R(t)$. Recall that $R(t)$ obeys the equation

$$(A-\gamma CK)R+R(A-\gamma CK)^\top +L = 0. $$ Taking the time derivative of both sides, and using again the short-hand $\ddto R = \dot R$ and $\ddto C = \dot C=  \ad_C\Omega_x$, we obtain

$$(A-\gamma CK)\dot R +\dot R(A-\gamma CK)^\top - \gamma R(\dot C K+C\dot K)-\gamma (K\dot C + \dot K C)R = 0.$$ Setting $S :=- R(\dot C K+C\dot K)$,  we can write explicitly

\begin{equation}
\label{eq:defdR}
\dot R = \gamma \int_0^\infty e^{(A-CK)t}(S+S^\top)e^{(A-CK)^\top t}.
\end{equation}
Gathering the relations above, we have the following expression for $X \cdot Y \cdot J$:
\begin{multline}
\label{eq:lemHessxyf}
X\cdot Y\cdot  J\\ = \gamma \left\{\<[[C,\Omega_x],M],\Omega_y\> +\<[C,\dot KRK],\Omega_y\>+\<[C,K\dot R K],\Omega_y\>+\<[C,KR\dot K)],\Omega_y\>\right\}
\end{multline} where $\dot K$ and $\dot R$ are given explicitly in~\eqref{eq:defdkint} and~\eqref{eq:defdR} respectively.
We now focus on the second term in Eqn.~\eqref{eq:defHess}. From Lemma~\ref{lem:levicivitanorma}, we now that $$\nabla_XY  =[C,[\Omega_x,\Omega_y]].$$ Let $C(t)$ be the curve in $\Sym(\Lambda)$ given by $C(t) = e^{t[\Omega_x,\Omega_y]}Ce^{-t[\Omega_x,\Omega_y]}$. Using the expression for the gradient of $J$ obtained in Theorem~\ref{th:equilgradientflow}, we get \begin{equation}
\label{eq:sectermHess}
\nabla_XY \cdot J = \gamma \< [\Omega_x,\Omega_y],M\>.\end{equation} Using the ad-invariance property of the normal metric,  the first term of~\eqref{eq:lemHessxyf} is equal to $\gamma \<[C,\Omega_x],[M,\Omega_y]\>$. Now recalling the expression of $d^2J$ given in~\eqref{eq:defHess}, we obtain the result using~\eqref{eq:lemHessxyf} and~\eqref{eq:sectermHess}.
\end{proof}
\begin{Lemma}\label{lem:covderivF}
Let $A$ be a stable matrix and $\gamma \geq 0$. Define $$F: \R^+ \times \Sym(n,p) \longmapsto T_C\Sym(n,p): (\gamma,C) \longmapsto [C,M]$$ where $K,R$ satisfy Eq.~\eqref{eq:eqKR} and $M=KRK$. The covariant derivative of $F$ at $(0,C)$ and with respect to its second argument is $$\nabla_X F=-\frac{1}{2}\left([M_0,[C,\Omega_x]]+[\Omega_x,[C,M_0]]\right).$$

\end{Lemma}
\begin{proof}
We need to evaluate $\nabla_\Omega F(0,C)$. Observe that $F(0,C)$ is a constant vector field as defined in~\eqref{eq:defconstanvf}. From Lemma~\ref{lem:levicivitanorma}, a short calculation yields  $$\nabla_X F = \frac{1}{2} [C,[\Omega_x,M_0]].$$ Using the Jacobi identity, the previous relation can expressed as$$\nabla_X F=-\frac{1}{2}\left([M_0,[C,\Omega_x]]+[\Omega_x,[C,M_0]]\right)$$ as announced.
\end{proof}

\lemmasiggeb*
\begin{proof}[Proof of Lemma~\ref{lem:calcsignHessgen}]
Recall that the dimension of $\Sym(n,p)$ is $d:=pn-p^2$. We first verify that for $p$ distinct integers  $1 \leq m_i \leq n$  summing to $m$, $m-\frac{p(p+1)}{2} \in \{0,1,\ldots,d\}$. Indeed, on the one hand the smallest value that $m$ can take is $1+2+\ldots+p=\frac{p(p+1)}{2}$. On the other hand, the largest value of $m$ is $(n-p+1)+(n-p+2)+\ldots+(n-1)+n$. This last expression is equal $p(n-p)+p(p+1)/2$.  This proves the  claim.

As before, we let $\Omega_{ij}$ be the skew-symmetric matrix with zero entries everywhere except for the $ij$th entry, which is $1$, and the $ji$th entry, which is $-1$ and we let $\Sigma_{ij}$ be the symmetric matrix with zeros everywhere except for the $ij$th and $ji$th entry, which are one. We have shown in Lemma~\ref{lem:isoorthobasis} that the tangent space of $\Sym(n,p)$ at $E$ is spanned by a basis with vectors $[E,\Omega_{ij}]$ where $i \in \mathcal{M}:=\{m_1,\ldots,m_p\}$ and $j$ is in the complement of ${\mathcal{M}}$ in $\{1,2,\ldots,n\}$, which we denoted $\overline{\mathcal{M}}$. We claim that this basis  diagonalizes the bilinear form $H_E$. To see this, first note that $$[D,\Omega_{ij}] = (d_i-d_j)\Sigma_{ij}.$$ Second, an easy calculation show that for $i>j$

$$[E,\Omega_{ij}]=\left\lbrace \begin{aligned}\Sigma_{ij}& \mbox{ if } i \in\mathcal{M}\mbox{and }j \in \overline{\mathcal{M}}\\
0& \mbox{ otherwise}\end{aligned}\right.$$ Because $\tr(\Sigma_{ij}\Sigma_{kl})=2$ if $i=j$ and $k=l$ and zero otherwise we conclude that
$$H_E(\Omega_{ij},\Omega_{kl})= 2 (d_i-d_j) \delta_{ik}\delta_{jl}.$$ The bilinear form is non-degenerate and because the $d_i$'s are distinct and sorted in decreasing order, i.e. $d_i-d_j >0$ if and only if $i>j$. Thus, the number of positive eigenvalues of $H_E$ is equal to the number of integer pairs $(i,j) \in \mathcal{M} \times \overline{\mathcal{M}}$ with $i>j$.  We can enumerate such pairs as follows, : for $i=m_1$, any $j\in \{1,\ldots, m_1-1\}$ is such that the above requirement on the pair $(i,j)$ is satisfied. There are $m_1-1$ such $j$'s. For $i=m_2$, the requirement holds for any $j \in \{1,\ldots, m_1-1,m_1+1, \ldots, m_2-1\}$. There are $m_2-2$ such $j$'s. Generally, for $i=m_l$, there are $m_1+m_2+\ldots+m_l-1-2-\ldots-l$ pairs passing the requirement.  Hence there is a total of $m -\frac{p(p+1)}{2}$ positive eigenvalues as announced.
\end{proof}

The following Lemma is used to show that the matrix $M_0$ used in the main part of the paper generically has distinct eigenvalues, and thus a unique basis of orthonormal eigenvectors.

\begin{Lemma}\label{lem:genQL}
Let $A\in\R^{n \times n}$ be a stable matrix and $Q, L$ be positive definite symmetric matrices. Let $K, R \in \R^{n\times n}$ be the unique positive definite solutions of
\begin{equation}
\label{lyapLemmgen}
\begin{aligned}
AK+ KA^\top +Q &=0\\
A R + RA^\top +L &=0 
\end{aligned}\end{equation}
Then generically for $Q,L$ positive definite, the matrix $M:=KRK$ has distinct eigenvalues.
\end{Lemma}

\begin{proof}
We first recall that since $A$ is stable, the Lyapunov equations in~\eqref{lyapLemmgen} have each a unique symmetric positive definite solution. We denote them  by $\sL(Q)$ and $\sL(L)$ respectively, i.e. $$\mathcal{L}(Q)=\int_0^\infty e^{At}Qe^{A^\top t}dt.$$  We let $S^+$ be the set of symmetric positive definite matrices of dimension $n$ and define  the map $F:S^+ \times S^+ \longmapsto S^+:(Q,L) \longmapsto M$ where $K=\sL(Q)$ and $R=\sL(L)$. The proof of the Lemma goes by showing that for a generic point $(Q,L) \in S^+ \times S^+$, $F$ is locally surjective, i.e. $F$ maps small enough neighborhoods of $(Q,L)$ onto neighborhoods of $F(Q,L)$. From there, the statement of the Lemma follows from a simple contradiction argument. Indeed, assume that $F$ is locally surjective but that there exists an open set $V \subset S^+\times S^+$ for which $F(V)$ only contains matrices with non-distinct eigenvalues. The set of matrices in $S^+$ which have non-distinct eigenvalues is of measure zero and thus for any pair $(Q,L)$ in $V$, $F$ is not locally surjective -- a contradiction.

The remainder of the proof is dedicated to showing that $F$ is generically locally surjective (g.l.s.). To this end, note that an open map is clearly g.l.s. and that the composition of generically locally surjective  maps is likewise  g.l.s. . To see that this last statement holds, assume that $f_1:M \longmapsto N $ and $f_2:N \longmapsto P$ are g.l.s. and let $f_3=f_2\circ f_1$. Let $C_1 \subset M$ (resp. $C_2 \subset N$) be the set of points where $f_1$ is not locally surjective (resp. $f_2$) and let $D = \{x \in M \mid f_1(x) \in C_2\}$. The sets $C_1$ and $C_2$ are of measure zero by assumption and  by the same argument as in the paragraph above, $D$ is of measure zero in $M$. Thus for $x \notin C_1 \cup D$, $f_3(x)$ is locally surjective and since $C_1 \cup D$ is of measure zero, $f_3$ is g.l.s. .

We now return to the main thread. Let $G:\R^{n\times n} \times \R^{n\times n} \longmapsto \R^{n\times n}:G(K,R)=M.$ Then we can write $F$ as the composition $F = G \circ (\sL(Q),\sL(L))$. The operator $\sL^{-1}(X)=AX+XA^\top$ is nothing more that the Lyapunov operator.  One can show that for $A$ stable, the Lyapunov operator is of full-rank (observe that its eigenvalues are pairwise sums of eigenvalues of $A$). Its inverse  $\sL$ is thus a full rank linear map  and consequently an open map. By a standard argument, one can show that the map $(Q,L)\longmapsto (\sL(Q),\sL(L)$ is also an open map. The map $G$ is a polynomial map and is clearly surjective. If we can show that $G$ is g.l.s., then $F$ is the composition of g.l.s. maps and is thus g.l.s. which proves the Lemma.

It thus remains to show that $G$ is g.l.s. To see this, first recall that at points $(K,R)$ in the domain of $G$ where its linearization $\frac{\partial G}{\partial x}$ is full rank, $G$ is locally surjective. Now assume that there is an open set $V$ in the domain of $G$ where its linearization is nowhere full rank. Then $\det(\frac{\partial F}{\partial x}\frac{\partial F}{\partial x}^\top)=0$ on the open set $V$ and because this determinant is a polynomial function, it is  zero everywhere. By Sard Theorem~\cite{Lee:2009dz}, the set $W$ over which the linearization of $G$ is not full rank is such that $G(W)$ has measure zero. But we have just shown that $W$ is the entire domain of $G$, which contradicts the fact that $G$ is surjective. This ends the proof of the Lemma.
\end{proof}

\end{document}